\documentclass{article}
\usepackage{amsfonts,amsmath,amssymb,amsthm}

\newtheorem{lemma}{Lemma}[section]
\newtheorem{theorem}[lemma]{Theorem}
\newtheorem{corollary}[lemma]{Corollary}
\newtheorem{proposition}[lemma]{Proposition}
\newtheorem{remark}[lemma]{Remark}
\newtheorem{definition}[lemma]{Definition}

\def\Carre#1#2{\vbox{
   \hrule height .#2pt
   \hbox{\vrule width .#2pt height #1pt \kern #1pt
      \vrule width .#2pt}
   \hrule height .#2pt}}

\def\LM#1{\hbox{\vrule width.2pt \vbox to#1pt{\vfill \hrule width#1pt
height.2pt}}}
\def\LL{{\mathchoice {\>\LM7\>}{\>\LM7\>}{\,\LM5\,}{\,\LM{3.35}\,}}}
\def\restr{{\LL}}

\def\Om{\Omega}
\def\e{\varepsilon}
\def\eps{\varepsilon}
\def\R{\mathbb{R}}

\def\Sp{\mathbb{S}}

\def\H{\mathcal{H}}
\def\HH{\mathcal{H}^{d-1}}

\def\Div{\textup{div}\,}
\def\supp{\textup{spt}\,}
\def\car#1{\raisebox{1pt}{$\chi$}_{#1}} 

\def\ov{\overline}
\def\onu{\ov{\nu}}

\def\half{\frac{1}{2}}

\begin{document}

\title{Fine properties of the subdifferential for a 
class of one-homogeneous functionals}
\author{A. Chambolle \footnote{CMAP, Ecole Polytechnique, CNRS,
        Palaiseau, France, email: antonin.chambolle@cmap.polytechnique.fr},
\and
        M. Goldman
        \footnote{Max-Planck-Institut f\"ur Mathematik
in den Naturwissenschaften, Inselstrasse 22, 04103 Leipzig,
Germany, email: goldman@mis.mpg.de, funded by a Von Humboldt PostDoc fellowship}
        \and M. Novaga
        \footnote{Dipartimento di Matematica, Universit\`a di Pisa,
        Largo B. Pontecorvo 5, 56127 Pisa, Italy, email: novaga@dm.unipi.it}
}
\date{}
\maketitle

\begin{abstract}
We collect some known results on the subdifferentials of a class of one-homogeneous functionals, which consist in anisotropic and nonhomogeneous variants of the total variation.
It is known that the subdifferential at a point is the divergence of some ``calibrating field''.
We establish new relationships between Lebesgue points of a calibrating field and regular points of the level surfaces of the corresponding calibrated function.
\end{abstract}


\section{Introduction}
In this note we recall some classical results on the structure of the subdifferential of first order one-homogeneous functionals, and
we give new regularity results which extend and precise previous work by G. Anzellotti \cite{Anzelotti,anzeltrace,anzeleuler}.

Given an open set $\Om\subset\R^d$  with Lipschitz boundary, and a function $u\in C^1(\Om)\cap BV(\Om)$, we consider the functional
\[
J(u)\ :=\ \int_\Om F(x,Du) 
\]
where $F:\Om\times\R^d\to [0,+\infty)$ is continuous, 
and $F(x,\cdot)$ is a smooth and uniformly convex positively
one-homogeneous functional on $\R^d$ 
for all $x\in \Om$. The functional $J$ can be canonically 
relaxed to the whole of $BV(\Om)$ (see \cite[Section 5.5]{AFP})
and we still write, in analogy with the notation commonly used for the total variation, 
$J(u)=\int_\Om F(x,Du)$ for $u\in BV(\Om)$, 
where $F(x,Du)$ is in general a Radon measure in $\Om$ defined by $F(x,Du):=F\left(x,\frac{Du}{|Du|}\right)|Du|$ with $\frac{Du}{|Du|}$ the Radon-Nikodym derivative of $Du$ with respect to $|Du|$.

Since $BV(\Om)\subset L^{d/(d-1)}(\Om)$, it is
natural to consider $J$ as a convex, l.s.c.~function on the whole of 
$L^{d/(d-1)}(\Om)$,
with value $+\infty$ when $u\not\in BV(\Om)$. 
In this framework,
for any $u\in L^{d/(d-1)}(\Om)$ such that $J(u)<+\infty$,
that is $u\in BV(\Om)$,
we can define the subdifferential of $J$
at $u$, in the duality $(L^{d/(d-1)},L^d)$, as
\begin{equation*}
\partial J(u)\ :=\ \left\{
g\in L^d(\Om)\,:\, J(v)\ge J(u)\,+\,\int_\Om g(v-u)\,dx\ \forall
v\in L^{d/(d-1)}(\Om) \right\}.
\end{equation*}
Notice that a function $g\in L^d(\Om)$ belongs to $\partial J(u)$ 
if and only if $u$ is a minimizer of the functional $J(v)-\int_\Om gv dx$ 
among all $v\in L^{d/(d-1)}(\Om)$.

The goal of this paper is to investigate the particular structure
of the functions $u$ and $g$, when $g\in \partial J(u)$. 
Let
\[
F^*(x,z) := \sup_{w\in\R^d} z \cdot w - F(x,w) 
\]
be the Legendre-Fenchel or convex conjugate of $F$.
Notice that $F^*$ takes values in $\{0,+\infty\}$, 
and $F^*(x,z)=0$ if and only if $F^\circ(x,z)\le 1$, where
$F^\circ$ denotes the convex polar of $F$ defined in~\eqref{polar} below.

Given $u\in L^{d/(d-1)}$, the functional $J(u)$ can also be expressed by duality as
\[
J(u)\ =\ \sup\left\{
-\int_\Om u \Div z \,dx\ :\ z\in C_c^\infty(\Om;\R^d)\,,\ F^*(x,z(x))=0
\ \forall x\in\Om
\right\}.
\]
It follows that a function $g\in\partial J(u)$ has necessarily the form 
$g=-\Div z$, for some vector field 
$z\in L^\infty(\Om;\R^d)$ with $F^*(x,z(x))=0$ a.e.~in $\Om$.
Since by a formal integration by parts one gets $$z\cdot Du= F(x,Du),$$ 
two natural questions arise: 
\begin{itemize}
\item in what sense is this relation true? 
\item can one assign a precise value to $z$ on the support of the measure $Du$? 
\end{itemize}

The first question has been answered by
Anzelotti in the series of papers~\cite{Anzelotti,anzeltrace,anzeleuler}. 
However, for the particular vector fields we are interested in,
we can be more precise and obtain pointwise properties of $z$
on the level sets of the function $u$. Indeed, we shall show
that $z$ has a pointwise meaning on all level sets of
$u$, up to $\H^{d-1}$-negligible sets (which can be
much more than $|Du|$-a.e.,
as illustrated by the function $u =\sum_{n=1}^{+\infty} 2^{-n}\chi_{(0,x_n)}$,
defined in the interval $(0,1)$, with $(x_n)$ a dense sequence in that
interval).

We will therefore focus on the properties of the vector fields
$z\in L^\infty(\Om,\R^d)$ such that $F^*(x,z(x))=0$ a.e. in $\Om$ 
and $g=-\Div z \in  L^d(\Om)$, and such that there exists a function
$u$ such that for any $\phi\in C_c^\infty(\Om)$,
\[
-\int_\Om\Div z \,u\phi\,dx\ =\ \int_\Om u\,z\cdot\nabla\phi\,dx
+\int_\Om \phi F(x,Du)\,.
\]
In particular, one checks easily that $u$  minimizes the functional 
\begin{equation}\label{minsubgrad}
\int_\Om F(x,Du) - \int_\Om g u \,dx
\end{equation}
among perturbations with compact support in $\Om$. 
Conversely, given $g\in L^d(\Om)$ with $\|g\|_{L^d}$ sufficiently small,
there exist functions $u$ which
minimize~\eqref{minsubgrad} under various types of boundary conditions,
and corresponding fields $z$.

This kind of functionals appears in many contexts including image processing and plasticity \cite{Andreu,Temam}.
Notice also that, by the Coarea Formula \cite{AFP}, it holds
\begin{equation*}
\int_\Om F(x,Du) - \int_\Om g u\,dx
= \int_\R \left( \int_{\partial^* \{u>s\}} F(x,\nu) \,d\H^{d-1}(x)
- \int_{\{u>s\}} g\,dx\right)\,ds\,,
\end{equation*}
where $\nu$ is the unit normal to $\{u>s\}$,
and one can show (see for instance~\cite{CMCF}) that
any level set of the form
$\{ u>s\}$ or $\{ u\ge s\}$ is a minimizer of the geometric functional 
\begin{equation}\label{geomfun}
E\mapsto \int_{\partial^* E} F(x,\nu)\,d\H^{d-1}(x) - \int_E g\,dx\,. 
\end{equation}
defined for sets $E$ of finite perimeter.
The canonical example of such functionals is given by the total variation, 
corresponding to $F(x,Du)=|Du|$. 
In this case, \eqref{geomfun} boils down to 
\begin{equation}\label{geomper}
P(E)-\int_E g \,dx.
\end{equation}

In \cite{BarGonTa}, it is shown that every set with finite perimeter in $\Om$
is a minimizer of \eqref{geomper} for some $g\in L^1(\Om)$.
If $F$ is even in $\nu$ and when $g\in L^d(\Om)$, the boundary $\partial E$ is only of class $C^{0,\alpha}$ out of a singular set (see \cite{AmbPao}). However, if $g\in L^p(\Om)$ with $p>d$, 
and $E$ is a minimizer of \eqref{geomfun},
then $\partial E$ is locally $C^{1,\alpha}$ for some $\alpha>0$, out of a closed singular set of zero $\H^{d-3}$-measure \cite{ASS,SchoenSimon} (some
regularity assumption on $F$ is required, see also Remark~\ref{remDu} below).
 Since the Euler-Lagrange equation of \eqref{geomfun} relates $z$ to the normal of $E$, 
understanding the regularity of $z$ is closely related to understanding the regularity of $\partial E$.

\smallskip

Our main result is that the Lebesgue points of $z$ correspond to regular points of $\partial \{u>s\}$ or $\partial \{u\ge s\}$ (Theorem \ref{zfixPL}), 
and that the converse is true in dimension $d\le 3$ (Theorem \ref{propLebesgue}).

\section{Preliminaries}

\subsection{BV functions}
We briefly recall the definition of function of bounded variation and set of finite perimeter. For a complete presentation we refer to \cite{AFP}.

\begin{definition}
Let $\Om$ be an open set of $\R^d$, we say that a function $u\in L^1(\Om)$ is a function of bounded variation if
\[\int_\Om |Du|:= \sup_{\stackrel{z \in \mathcal{C}^1_c(\Om)}{|z|_\infty \leq 1}} \int_{\Om} u \ \Div z \ dx< +\infty.\]
We denote by $BV(\Om)$ the set of functions of bounded variation in $\Om$ (when $\Om=\R^d$ we simply write $BV$ instead of $BV(\R^d)$). 
We say that a set $E\subset \R^d$ is of finite perimeter if its characteristic function $\chi_E$ is of bounded variation and denote its perimeter in an open set $\Om$ by $P(E,\Om):=\int_{\Om} |D \chi_E|$, and write simply $P(E)$ when $\Om=\R^d$.
\end{definition} 

\begin{definition}
Let $E$ be a set of finite perimeter and let $t\in [0;1]$. We define
\[
E^{(t)}:= \left\{ x \in \R^d \,:\, \lim_{r\downarrow 0} \frac{|E\cap B_r(x)|}{|B_r(x)|}=t\right\}.
\]
We denote by $\partial E:= \left(E^{(0)}\cup E^{(1)}\right)^c$ the measure theoretical boundary of $E$.
Let $\supp(|D \chi_E|)$ be the support of the measure $|D \chi_E|$: 
we define the reduced boundary of $E$ by:
\[
\partial^* E:= \left\{ x \in \supp(|D \chi_E|) \,: \, \nu^E(x):= \lim_{r \downarrow 0} \frac{ D \chi_E (B_r(x))}{|D \chi_E|(B_r(x))} \; \textrm{exists and } \; |\nu^E(x)|=1 \right\}\subset E^{\left(\frac 1 2\right)}.
\] 
The vector $\nu^E(x)$ is the measure theoretical inward normal to the set $E$. 
\end{definition} 
\begin{proposition}
If $E$ is a set of finite perimeter then $D\chi_E=\nu^E\  \HH \LL\partial^* E$, $P(E)=\HH(\partial^* E)$ and $\HH(\partial E \setminus \partial^* E)=0$.
\end{proposition}

\begin{definition}
 We say that $x$ is an approximate jump point of $u\in BV(\Om)$ 
if there exist  $\xi \in \Sp^{d-1}$ 
and distinct $a, b \in \R$
 such that
\[ \lim_{\rho \to 0} \frac{1}{|B_\rho^+(x,\xi)|} \int_{B_\rho^+(x,\xi)} |u(y)-a| \ dy =0 \quad \textrm{ and } \quad  \lim_{\rho \to 0} \frac{1}{|B_\rho^-(x,\xi)|} \int_{B_\rho^-(x,\xi)} |u(y)-b| \ dy =0,\]
where $B_\rho^\pm(x,\xi):= \{ y \in B_\rho(x) : \pm (y-x)\cdot{\xi}>0 \}.$ Up to a permutation of $a$ and $b$ and a change of sign of $\xi$, this characterize the triplet $(a,b,\xi)$ which is then denoted by $(u^+,u^-, \nu_u)$. The set of approximate jump points is denoted by $J_u$. 
\end{definition}

The following proposition can be found in \cite[Proposition 3.92]{AFP}.
\begin{proposition}\label{Ju}
 Let $u\in BV(\Om)$. Then, defining
\[\Theta_u:=\{ x\in \Om \, : \, \liminf_{\rho \to 0} \rho^{1-d} |Du|(B_\rho(x))>0\},\]
there holds $J_u\subset \Theta_u$ and $\H^{d-1}(\Theta_u \setminus J_u)=0$.
\end{proposition}

\subsection{Anisotropies}
Let $F(x,p):\Om \times \R^d\to [0,+\infty)$ be a continuous functions, which is 
convex and positively one-homogeneous in the second variable:
$F(x,\lambda p)=\lambda F(x,p)$ for all $\lambda>0,x,p$;
and such that there exists $c_0>0$ with
\[c_0 |p|\le F(x,p)\le \frac{1}{c_0} |p| \qquad \forall (x,p)\in \R^d\times \R^d.\]
We say that $F$ is uniformly elliptic if for some $\delta>0$, the function 
$p\mapsto F(p)-\delta |p|$ is still a convex function.  We define the polar function of $F$ by
\begin{equation}\label{polar}
F^\circ(x,z):=\sup_{\{F(x,p)\le1\}} z\cdot p
\end{equation}
so that $(F^\circ)^\circ=F$.  It is easy to check that 
$[F(x,\cdot)^2/2]^*= F^\circ(x,\cdot)^2/2$, where as before the $^*$ denotes
the convex conjugate with respect to the second variable.
In particular, if differentiable,
$F(x,\cdot)\nabla_p F(x,\cdot)$ and $F^\circ(x,\cdot)\nabla_z F^\circ(x,\cdot)$ 
are inverse monotone operators.
Also, one has that $F^*(x,z)=0$ if and only if $F^\circ(x,z) \le 1$,
and $F^*(x,z)=+\infty$ else. 

If $F(x,\cdot)$ is differentiable then, for every $p\in \R^d$, 
\[
F(x,p)=p\cdot \nabla_p F(x,p) \qquad {\rm (Euler's\ identity)}
\] 
and
\[ z\in \{F^\circ(x,\cdot)\le 1\} \textrm{ with } p\cdot z=F(x,p)\ \Longleftrightarrow \  z=\nabla_p F(x,p).\]
If $F$ is elliptic and of class $\mathcal{C}^2(\R^d \times \R^d\setminus\{0\})$, then $F^\circ$ is also elliptic and $\mathcal{C}^2(\R^d \times \R^d\setminus\{0\})$. We will then say that $F$ is a smooth elliptic anisotropy. 
Observe that, in this case, the function $F^2/2$ is also uniformly $\delta^2$-convex
(this follows from the inequalities $D^2F(x,p)\ge \delta/|p|(I-p\otimes p/|p|^2)$
and $F(x,p)\ge \delta |p|$). In particular, for every $x,y,z \in \R^d$, there holds 
\begin{equation}\label{strictconvF}F^2(x,y)-F^2(x,z)\ge 2\left(F(x,z)\nabla_p F(x,z)\right)\cdot (y-z) + \delta^2|y-z|^2,
\end{equation}
and a similar inequality holds for $F^\circ$. 
We refer to \cite{schneider} for general results on convex norms and convex bodies.

\subsection{Pairings between measures and bounded functions}
Following \cite{Anzelotti} we define a generalized trace $[z,Du]$ 
for functions $u$ with bounded variation and bounded vector fields $z$ with divergence in $L^d$.
\begin{definition}[Anzelotti~\cite{Anzelotti}]
Let $\Om$ be an open set with Lipschitz boundary,  $u\in BV(\Om)$ and
 $z \in L^{\infty}(\Om,\R^d)$ with $\Div z\in L^d(\Om)$. 
We define the distribution $[z,Du]$ by 
 \[ \langle [z,Du], \psi \rangle= -\int_\Om u \,\psi\, \Div z \, dx - \int_\Om u \,  z \cdot \nabla \psi  \, dx
 \qquad \forall \psi \in \mathcal{C}^\infty_c(\Om). 
 \]
\end{definition}

\begin{proposition}[Anzelotti~\cite{Anzelotti}]
The distribution $[z,Du]$ is a bounded Radon measure on $\Om$ and if $\nu$ is the inward unit normal to $\Om$, there exists a function $[z,\nu]\in L^\infty(\partial \Om)$ such that the generalized Green's formula holds,
 \[ \int_\Om [z,Du]=- \int_\Om u \Div z \, dx -\int_{\partial \Om} [z, \nu] u\,d \HH.\]
 The function $[z,\nu]$ is the generalized (inward) normal trace of $z$ on $\partial \Om$. 
\end{proposition} 
  
Given $z \in L^{\infty}(\Om;\R^d)$, with $\Div z\in L^d(\Om)$, we can
also define the generalized trace of $z$ on $\partial E$, where $E$ is a set of locally finite perimeter. 
Indeed, for every bounded open set $\Om$ with Lipschitz boundary, we can define as above the measure $[z,D\chi_E]$ on $\Om$. 
Since this measure is absolutely continuous with respect to $|D\chi_E|= \HH \restr \partial^* E$ we have 
\[
[z,D\chi_E]= \psi_z(x) \HH\restr \partial^* E
\]
with $\psi_z \in L^{\infty}(\partial^* E)$ independent of $\Om$. 
We denote by $[z,\nu^E]:=\psi_z$ the generalized (inward) normal trace of $z$ on $\partial E$. 
If $E$ is a bounded set of finite perimeter, by taking $\Om$ strictly containing $E$,
we have the generalized Gauss-Green Formula 
\[\int_E \Div z \, dx =-\int_{\partial^* E} [z,\nu^E] d \HH.\]
Anzellotti proved the following alternative definition of $[z,\nu^E]$ \cite{anzeltrace,anzeleuler}

\begin{proposition}[Anzelotti~\cite{anzeltrace,anzeleuler}]
\label{convcylinder}
Let $(x,\alpha)\in \R^d\times \R^d\setminus\{0\}$.
For any $r>0$, $\rho>0$ we let 
$$
C_{r,\rho}(x,\alpha):=\{ \xi \in \R^d \, : \,
\left|(\xi-x)\cdot \alpha\right| <r,
\left|(\xi-x)-[(\xi-x)\cdot \alpha] \alpha\right|<\rho\}.
$$
There holds
\[ [z,\alpha](x) =\lim_{\rho\to 0}\lim_{r\to 0} \  \frac{1}{2r \omega_{d-1} \rho^{d-1}} \int_{C_{r,\rho}(x,\alpha)} z \cdot \alpha\, dy \]
where $\omega_{d-1}$ is the volume of the unit ball in $\R^{d-1}$.
\end{proposition}

\section{The subdifferential of anisotropic total variations}

\subsection{Characterization of the subdifferential}

The  following characterization of the  subdifferential of $J$ is classical and readily follows for example from the representation
formula~\cite[(4.19)]{DMBou}.
\begin{proposition}\label{defcalibration}
Let $F$ be a smooth elliptic anisotropy and $g\in L^d(\Om)$ then $u$ is a local minimizer of \eqref{minsubgrad} if and only if there exists $z\in L^\infty(\Om)$ with $\Div z =g$, $F^*(x,z(x))=0 $ a.e. and 
\[ [z,Du]= F(x,Du),\]
where the equality holds in the sense of measures.
Moreover, for every $t\in \R$, for the set $E=\{u>t\}$ there holds $[z,\nu^E]=F(x,\nu^E) $ $\HH$-a.e. on $\partial E$. We will say that such a vector field is a calibration of the set $E$ for the minimum problem~\eqref{geomfun}.
\end{proposition}

\begin{remark}\rm
In \cite{Anzelotti}, it is proven that if $z_\rho(x):=\frac{1}{|B_\rho(x)|}\int_{B_\rho(x)} z(y) \ dy$, then $z_\rho\cdot\nu^E$  weakly* converges to $[z, \nu^E]$ in $L^\infty_{loc}(\HH\LL\partial^* E)$. Using \eqref{strictconvF} it is then possible to prove that if $z$ calibrates $E$ then $z_\rho$ converges to $\nabla_p F(x,\nu^E)$ in $L^2(\HH\LL \partial^*E)$ yielding that up to a subsequence,
$z_{\phi(\rho)}$ converges $\HH$-a.e. to $\nabla_p F(x, \nu^E)$. Unfortunately this is still a very weak statement since it is a priori impossible to recover from this the convergence of the full sequence $z_\rho$.
\end{remark}

The main question we want to investigate now is whether we can give a classical meaning to $[z,\nu^E]$ (that is understand if $[z,\nu^E]=z\cdot \nu^E$). We observe that a priori the value of $z$ is not well defined on $\partial E$ which has zero Lebesgue measure.

We let $S:=\supp(Du)\subset\Om$ be the support of the measure
$Du$, that is, the smallest closed set 
in $\Om$ such that $|Du|(\Om\setminus S)=0$. We will show that essentially
in $S$, $z$ is well-defined, as soon as $g\in L^d(\Om)$.

The next result is classical, for a proof we refer to \cite{Massari,GoMa}.
\begin{lemma}[Density estimate]\label{lemdens}
There exists $\rho_0>0$ (depending on $g\in L^d(\Om)$)
and a constant $\gamma>0$ (which depends only on $d$), such that for
any $B_\rho(x)\subset \Om$ 
with $\rho\le\rho_0$, and any level set $E$ of $u$ (that is,
$E\in \left\{\{u>s\},\{u\ge s\},\{u<s\},\{u\le s\},\, s\in \R\right\}$), if
$|B_\rho(x)\cap E|<\gamma |B_\rho(x)|$ 
then $|B_{\rho/2}(x)\cap E|=0$.
As a consequence, $E^{0}$ and $E^{1}$ are open, $\partial E$ is the topological
boundary of $E^{1}$, 
and (possibly changing slightly $\gamma$)
if $x\in \partial E$, then $\H^{d-1}(\partial E\cap B_\rho(x))
\ge \gamma \rho^{d-1}$.
\end{lemma}
This result is not true anymore if $g\not\in L^d(\Om)$~\cite{GoMa}.
If $\partial \Om$ is Lipschitz, it is true up to the boundary.
\medskip

\begin{corollary}
It follows that $u\in L^\infty_{loc}(\Om)$\, and $u\in C(\Om\setminus \Theta_u)$.
\end{corollary}
\begin{proof}
 For any ball $B_\rho(x)\subset \Om$ and
 $\inf_{B_{\rho/2}(x)}u<a<b<\sup_{B_{\rho/2}(x)}u$, one has
\[
+\infty\,>\, |Du|(B_\rho(x))\,\ge \,\int_{a}^b P(\{u>s\},B_\rho(x))\,ds
\,\ge\, (b-a)\gamma\left(\frac{\rho}{2}\right)^{d-1}\,,
\]
so that $osc_{B_{\rho/2}(x)}(u)$ must be bounded and thus $u\in L^\infty_{loc}(\Om)$. Moreover, if $x \in \Om\setminus \Theta_u$ we find that 
 $\lim_{\rho \to 0} osc_{B_{\rho}(x)}(u)=0$ so that $u$ is continuous at the point $x$.
\end{proof}

We remark that if sets $(E_n)_n$ satisfy the density estimate of Lemma~\ref{lemdens}
and converge in $L^1$ to some limit set, then one easily deduces that
the convergence also holds in the Hausdorff (or Kuratowski, is the sets
are unbounded) sense. Applying this principle to the level sets of $u$, 
we find that all points in the support
of $Du$ must be on the boundary of a  level set of $u$:
\begin{proposition}\label{proplevelset}
For any $x\in S$, there exists $s\in \R$ such that either $x\in \partial\{u >s\}$
or $x\in \partial \{u\ge s\}$.
\end{proposition}
\begin{proof} First, if $x\not\in S$ then $|Du|(B_\rho(x))=0$
for some $\rho>0$ and clearly $x$ cannot be on the boundary of a level
set of $u$. On the other hand, if $x\in S$, then for any ball $B_{1/n}(x)$
($n$ large) there is a level $s_n$ (uniformly bounded) with
$\partial \{u>s_n\}\cap B_{1/n}(x)\neq \emptyset$ and by Hausdorff convergence,
we deduce that either $x\in\partial\{u>s\}$ or  $x\in\partial \{u\ge s\}$
where $s$ is the limit of the sequence $(s_n)_n$ (which must actually converge).
\end{proof}
\medskip
The following stability property is classical (see e.g. \cite{CGN}).
\begin{proposition}\label{stabilitycalib}
Let $E_n$ be 
local minimizers of~\eqref{geomfun}, with a function $g=g_n\in L^d(\Om)$,
and converging in the $L^1$-topology to a set $E$.
Assume that the sets $E_n$ are calibrated by $z_n$ (see Prop.~\ref{defcalibration}),
that $z_n\stackrel{*}{\rightharpoonup} z$ weakly-$*$ in $L^\infty$
and $g_n\to g=-\Div z\in L^d(\Om)$, in $L^1(\Om)$ as $n\to\infty$.
Then $z$ calibrates $E$, which is
thus also a minimizer of~\eqref{geomfun}.
\end{proposition}
Let us observe that, if $z_n\stackrel{*}{\rightharpoonup}z$
and $F^\circ(x,z_n(x))\le 1$ a.e. in $\Om$, then in the limit one still gets
$F^\circ(x,z(x))\le 1$ a.e. in $\Om$, thanks to the continuity of $F$
and the convexity in the second variable.


\subsection{The Lebesgue points of the calibration.}

The next result shows that the regularity of the calibration $z$
implies some regularity of the calibrated set.

\begin{theorem}\label{zfixPL}
Let $E=\{u>t\}$ or $E=\{u\ge t\}$, and
let $\bar x\in \partial E$ be a Lebesgue point of $z$.
Then, $\bar x  \in \partial^*E$ and
\begin{equation}\label{zeqnu}
z(\bar x)=\nabla_p F(\bar x,\nu^E(\bar x)).
\end{equation}
\end{theorem}

\begin{proof}
We reason as in \cite[Th. 4.5]{CGN} and let $z_\rho(y) := z(\bar x+\rho y)$.
Since $\bar x$ is a Lebesgue point of $z$, we have
that $z_\rho\to\bar z$ in $L^1(B_R)$, hence
also weakly-$*$ in $L^\infty(B_R)$ for any $R>0$,
where $\bar z\in \R^d$ is a constant vector.

We let $E_\rho := (E-\bar x)/\rho$ and  $g_\rho(y):=g(\bar x + \rho y)$ (so that $\Div z_\rho=\rho g_\rho$). 
Observe that $E_\rho$ minimizes 
\[
\int_{\partial^* E_\rho\cap B_R} F(\bar x+\rho y,\nu^{E_\rho}(y))\,d\HH(y)
+ \rho\int_{E_\rho\cap B_R} g_\rho(y)\,dy\,,
\]
with respect to compactly supported
perturbations of the set (in the fixed ball $B_R$).
Also,
\[
\|\rho g_\rho\|_{L^d(B_R)}\ = \ \|g\|_{L^d(B_{\rho R})}
\ \stackrel{\rho\to 0}{\longrightarrow}\ 0\,.
\]
By Lemma~\ref{lemdens}, the sets $E_\rho$ (and the boundaries $\partial E_\rho$)
satisfy uniform density bounds, and hence are compact with respect
to both local $L^1$ and Hausdorff convergence.

Hence, up to extracting a subsequence, we can assume that $E_\rho\to \bar E$, with $0\in \partial \bar E$.
Proposition~\ref{stabilitycalib} shows that $\bar z$ is a calibration
for the energy $\int_{\partial \bar E\cap B_R} F(\bar x, \nu^{\bar E}(y))\,d\HH(y)$,
and that $\bar E$ is a minimizer calibrated by $\bar z$.

It follows that $[\bar z,\nu^{\bar E}]=F(\bar x,\nu^{\bar E}(y))$ for
$\HH$-a.e.~$y$ in $\partial \bar E$, but since $\bar z$ is a constant,
we deduce that $\bar E=\{ y\cdot \bar\nu\ge 0\}$ with
$\bar \nu/F(\bar x,\bar \nu)=\nabla_p F^\circ(\bar x,\bar z)$\footnote{
We use here that $F(\bar x,\cdot)\nabla F(\bar x,\cdot)=
[F^\circ(\bar x,\cdot)\nabla F^\circ(\bar x,\cdot)]^{-1}$, so that
$\bar z=\nabla F(\bar x,\nu^{\bar E}(y))$ implies both
$F^\circ(\bar x,\bar z)=1$ and
$\nu^{\bar E}(y)/F(\bar x,\nu^{\bar E})(y)=\nabla F^\circ(\bar x,\bar z)$}.
In particular the limit $\bar E$
is unique, hence we obtain the global convergence of $E_\rho\to \bar E$, without passing to a subsequence.

We want to deduce that $\bar x\in\partial^* E$, with
$\nu^E(\bar x)=F(\bar x,\nu^{E}(\bar x))\nabla_p F^\circ(\bar x,\bar z)$, which is equivalent to \eqref{zeqnu}. 
The last identity is obvious from
the arguments above, so that we only need to show that
\begin{equation}\label{reducbound}
\lim_{\rho\to 0} \frac{ D\chi_{E_\rho}(B_1) }{ |D\chi_{E_\rho}|(B_1) }\ =\ \bar \nu\,.
\end{equation}
Assume we can show that
\begin{equation}\label{convenergy}
\lim_{\rho\to 0} |D\chi_{E_\rho}|(B_R)\ =\ |D\chi_{\bar E}|(B_R)
\ \left(\ =\ \omega_{d-1}R^{d-1}\right)
\end{equation}
for any $R>0$, then for any $\psi\in C_c^\infty(B_R;\R^d)$ we would get
\begin{multline*}
\frac{1}{|D\chi_{E_\rho}|(B_R)}\int_{B_R} \psi\cdot D\chi_{E_\rho}
\ =\ -\frac{1}{|D\chi_{E_\rho}|(B_R)}\int_{B_R\cap E_\rho} \Div \psi(x)\,dx
\\ \longrightarrow\ -\frac{1}{|D\chi_{\bar E}|(B_R)}\int_{B_R\cap \bar E} \Div \psi(x)\,dx
\ =\ \frac{1}{|D\chi_{\bar E}|(B_R)}\int_{B_R} \psi\cdot D\chi_{\bar E}
\end{multline*}
and deduce that the  
measure $D\chi_{E_\rho}/(|D\chi_{E_\rho}|(B_R))$  weakly-$*$ converges to 
$D\chi_{\bar E}/(|D\chi_{\bar E}|(B_R))$.  
Using again \eqref{convenergy}), we then obtain that 
\begin{equation}\label{aeR}
\lim_{\rho\to 0} \frac{ D\chi_{E_\rho}(B_R) }{ |D\chi_{E_\rho}|(B_R) }\ =\ \bar \nu
\end{equation}
for almost every $R>0$.
Since $D\chi_{E_\rho}(B_{\mu R})/(|D\chi_{E_\rho}|(B_{\mu R}))=D\chi_{E_{\rho/\mu}}(B_{R})/(|D\chi_{E_{\rho/\mu}}|(B_{R}))$ for any $\mu>0$, 
\eqref{aeR} holds in fact for any $R>0$ and \eqref{reducbound} follows, so that
$\bar x\in \partial^* E$.

It remains to show~\eqref{convenergy}. First, we observe that,
by minimality of $E_\rho$ and $\bar E$ plus the Hausdorff convergence of $\partial E_\rho$ in balls,
we can easily show the convergence of the energies
\begin{multline*}
\lim_{\rho\to 0} \int_{\partial E_\rho\cap B_R}  F(\bar x+\rho y,\nu^{E_\rho}(y))\,d\HH(y)
+ \rho\int_{E_\rho\cap B_R} g_\rho(y)\,dy
\\ =\ \int_{\partial \bar E\cap B_R}
F(\bar x,\nu^{\bar E}(y))\,d\HH(y)
\end{multline*}
and, by the continuity of $F$,
\begin{equation}\label{convenergyrho}
\lim_{\rho\to 0} \int_{\partial E_\rho\cap B_R}  F(\bar x,\nu^{E_\rho}(y))\,d\HH(y)
\ =\ \int_{\partial \bar E\cap B_R}
F(\bar x,\nu^{\bar E}(y))\,d\HH(y)\,.
\end{equation}
Then, \eqref{reducbound} follows from Reshetnyak's continuity
theorem where, instead of using the Euclidean norm as reference norm, we use the uniformly convex function $F(\bar x, \cdot)$ and the convergence of the measures $F(\bar x, D\chi_{E_\rho})$ to $F(\bar x,D \chi_{\bar E})$ (see  \cite{Resh,CGN}). 

\end{proof}

\begin{corollary}
For any $x\in S$ let $E^x \in \left\{\{u>u(x)\},\{u\ge u(x)\}\right\}$
be the upper level set of $u$ such
that $x\in\partial E^x$. Then, the equality
\begin{equation}\label{euleru}
 z(x)=\nabla_p F\left(x,\frac{D\chi_{E^x}}{|D\chi_{E^x}|}(x)\right)
\end{equation}
holds Lebesgue a.e. in $S=\supp(Du)$.
\end{corollary}

\begin{remark}\rm 
In the inhomogeneous isotropic case $F(x,p)=a(x)|p|$, with $a(\cdot)$ periodic, 
a similar result has been proved by 
Auer and Bangert in \cite[Th. 4.2]{AB}. As a consequence they obtain
differentiability properties of the so-called stable norm associated to 
the functional $J$ (see also \cite{CGN} for the anisotropic 
version of their result).   
\end{remark}

In dimension $2$ and $3$ we can also show the reverse implication, proving that regular points of the boundary correspond to Lebesgue points of the calibration. The idea is to show that the parameters $r,\,\rho$ in Proposition \ref{convcylinder} can be taken of the same order. 

\begin{theorem} \label{propLebesgue}
Assume the dimension is $d=2$ or $d=3$.
Let $x,s$ be as  in Proposition~\ref{proplevelset}, $E$ be a minimizer of \eqref{geomfun} and assume
$x\in \partial^*E$. Then $x$ is a Lebesgue point of $z$ 
and \eqref{zeqnu} holds at $x$.
\end{theorem}

\begin{proof}
We divide the proof into two steps.

{\it Step 1.} We first consider anisotropies $F$ which are not depending on the $x$ variable.  Without loss of generality we assume $x=0$.
By assumption, there exists the limit
\begin{equation}
\onu \ :=\ \lim_{\rho\to 0} \frac{ D\car{E}(B_\rho(0))}{|D\car{E}|(B_\rho(0))|}
\end{equation}
and, without loss of generality, we assume that it coincides with
the vector $e_d$ corresponding to the last coordinate of $y\in \R^d$. 

Also, if we let $E_\rho:= E/\rho$, the sets $E_\rho$, $E_\rho^c$,
$\partial E_\rho$
converge in $B_1(0)$, in the Hausdorff sense (thanks to the uniform
density estimates), respectively to $\{y_d\ge 0\}$, $\{y_d=0\}$, $\{y_d\le 0\}$.
We also let $z_\rho(y) := z(\rho y)$ and $g_\rho(y) := g(\rho y)$,
in particular $-\Div z_\rho=\rho g_\rho$. We let
 \begin{equation}\label{defomega}
 \omega(\rho)\ =\  \sup_{x\in\Om}\|g\|_{L^d(B_\rho(x)\cap\Om)}\,
 \end{equation}
which is continuously increasing and goes to $0$ as $\rho\to 0$,
since $|g|^d$ is equi-integrable.

We introduce the following notation: a point in $\R^d$ is denoted by
$y=(y',y_d)$, with $y'\in \R^{d-1}$. We let $D_s:=\{|y'|\le s\}$, $\bar z:= \nabla F(\onu)$ and 
$D_s^t = \{ D_s+ \lambda \bar z: |\lambda|\le t\}$ and denote with $\partial D_s$ the relative boundary
of $D_s$ in $\{y_d=0\}$.

We choose $s\le 1$, $0< t\le s$, ($t$ is chosen small enough
so that $D_s^t\subset B_1(0)$, that is $t < (1/|\bar z|)\sqrt{1-s^2}$).
We integrate in $D_s^t$ the divergence $\rho g_\rho=-\Div z_\rho=\Div(\bar z-z_\rho)$
against the function $(2\car{E}-1)t-\frac{ \onu\cdot y}{F(\onu)} $, which vanishes
for $y_d=\pm t F(\onu) $ if $\rho$ is small enough (given $t>0$), so 
that $\partial E_\rho\cap B_1(0)\subset \{|y_d|\le t F(\onu)\}$. For $y$ on the lateral boundary of the cylinder $D_s^t$, let $\xi(y)$ be the internal normal to $\partial D_s+(-t,t)\bar z$ at the point $y$. 
Using the fact that $z_\rho$ is a calibration for $E_\rho$,
we easily get that for almost all $s$,
\begin{multline}\label{bigone}
\int_{D_s^t} \rho g_\rho\left((2\car{E}-1)t-\frac{ \onu\cdot y}{F(\onu)}\right)\,dy
\\ =\ 
\int_{\partial D_s+(-t,t)\bar z} \left((2\car{E}-1)t-\frac{ \onu\cdot y}{F(\onu)}\right)
[(\bar z-z_\rho),\xi]\,d\H^{d-1}
\\-\,2t\int_{\partial E_\rho\cap D_s^t} 
\left(\bar z\cdot \nu^{E_\rho}-F(\nu^{E_\rho})\right)\,d\H^{d-1}
\,+\, \int_{D_s^t} \left(1-\frac{z_\rho\cdot\onu}{F(\onu)}\right)\,dy\,.
\end{multline}
Now since $F^\circ(\nabla F(\onu))=1$, there holds $\bar z\cdot \nu^{E_\rho}-F(\nu^{E_\rho})\le 0$ and using that $\bar z\cdot \xi(y)=0$ on $\partial D_s+(-t,t)\bar z$, we get 
\begin{multline}\label{bigtwo}
\int_{D_s^t} \left(1-\frac{z_\rho\cdot\onu}{F(\bar \nu)}\right)\,dy\le\ 
\int_{D_s^t} \rho g_\rho\left((2\car{E}-1)t-\frac{ \onu\cdot y}{F(\onu)}\right)\,dy
\\\int_{\partial D_s+(-t,t)\bar z} \left((2\car{E}-1)t-\frac{ \onu\cdot y}{F(\onu)}\right)
z_\rho\cdot \xi\,d\H^{d-1}\,.
\end{multline}

We claim that for $|\xi|\le 1$ with $\xi\cdot \bar z=0$, there holds
\begin{equation}\label{star}
 (\xi\cdot z_\rho)^2\le C(F(\onu)-\onu\cdot z_\rho)
\end{equation}
 Since 
\[(\xi\cdot z_\rho)^2\le |z_\rho|^2-[z_\rho\cdot (\bar z/|\bar z|)]^2\]
it is enough to prove
\[|z_\rho|^2-[z_\rho\cdot (\bar z/|\bar z|)]^2\le  C(F(\onu)-\onu\cdot z_\rho).\]

Using that $\onu/F(\onu)=\nabla F^\circ (\bar z)$, from \eqref{strictconvF} applied to $F^{\circ}$ together with $F^\circ(\bar z)=1\ge F^{\circ}(z_\rho)$, we find
\[(F(\onu)-\onu\cdot z_\rho)=F(\onu) (1-z_\rho\cdot \nabla F^\circ(\bar z))\ge C|z_\rho-\bar z|^2.\]
which readily implies  \eqref{star}. We thus have
\begin{multline}\label{big4}
\int_{\partial D_s+(-t,t)\bar z}
\left((2\car{E_\rho} -1) t -\frac{\onu\cdot y}{F(\onu)} \right)
\,\left(z_\rho\cdot \xi\right)\,d\H^{d-1}
\\ \le\,
2 C\sqrt{F(\onu)} t \int_{\partial D_s+(-t,t)\bar z}
\sqrt{1-\frac{z_\rho\cdot\onu}{F(\onu)}}\;d\H^{d-1}
\\ \le\,
2CF(\onu) t\sqrt{t} \left(
\int_{\partial D_s+(-t,t)\bar z} 
 \left(1-\frac{z_\rho\cdot\onu}{F(\onu)}\right) \,d\H^{d-1}
\right)^\half \sqrt{\H^{d-2}(\partial D_s)}\,.
\end{multline}
Now, we also have
\begin{multline}\label{big5}
\rho\int_{D_s^t}\left((2\car{E_\rho} -1) t -\frac{\onu\cdot y}{F(\onu)}\right)
g_\rho\,dy\,\le\,2t\rho^{1-d}\int_{D_{\rho s}^{\rho t}}g \,dy
\\
\le\, 2t \rho^{1-d}\|g\|_{L^d(B_{\rho s}(0))}  |D_{\rho s}^{\rho t}|^{1-1/d}
\,\le\, c t^{2-1/d} s^{d-2+1/d} \omega(\rho s)
\end{multline}
where here, $c=2\H^{d-1}(D_1)^{1-1/d}$,
and $\omega$ is defined in~\eqref{defomega}.

We choose $a<1$, close to $1$, and  $t\in (0,(1/|\bar z|)\sqrt{1-a^2})$.
If $\rho>0$ is small enough
(so that $\partial E_\rho\cap B_1$ is in $\{ |y_d|\le t F(\onu)\}$),
letting $f(s):=\int_{D^t_s} \left(1-\frac{z_\rho\cdot\onu}{F(\onu)}\right)dy$,
we deduce from~\eqref{bigtwo}, \eqref{big4} and~\eqref{big5}
that for a.e. $s$ with $t\le s\le a$, one has (possibly increasing
the constant $c$)
\begin{equation}\label{finalestimate}
f(s)^2\ \le\  c\left( s^{d-2} t^3  f'(s)
\,+\,t^{4-2/d} s^{2d-4+2/d} \omega(\rho s)^2\right)\,.
\end{equation}
Unfortunately, this estimate does not seem to  give much information for $d>3$.
It seems it allows to conclude only whenever $d\in \{2,3\}$.
Since the case $d=2$
is simpler, we focus on $d=3$. Estimate~\eqref{finalestimate} becomes
\begin{equation}\label{finalestimate3D}
f(s)^2\ \le\  c\left( s t^3  f'(s)
\,+\,  t^{10/3} s^{8/3} \omega(\rho s)^2\right)\,.
\end{equation}
Given $M>0$, we fix a value $t>0$ such that $\log(a/t)\ge c M$.
If $\rho$ is chosen small enough, then $\partial E_\rho\cap B_1(0)
\subset \{ |y_d| < t F(\onu)\}$, and~\eqref{finalestimate3D} holds.
It yields (assuming $f(t)>0$, but if not, then the proposition
is proved)
\begin{equation}\label{ff}
-\frac{f'(s)}{f(s)^2}\,+\,\frac{1}{c t^3} \frac{1}{s} \ \le\ 
ct^{1/3} s^{5/3} \frac{\omega(\rho s)^2}{f(s)^2}
\ \le\ 
ct^{1/3} s^{5/3} \frac{\omega(a \rho )^2}{f(t)^2}
\end{equation}
where we have used the fact that $t\le s\le a$ and  $f,\omega$
are nondecreasing.
Integrating \eqref{ff} from $t$ to $a$, after multiplication by $t^3$ we obtain
\[
\frac{t^3}{f(a)} - \frac{t^3}{f(t)} + \frac{\log(a/t)}{c}
\ \le\ \frac{3c}{8} t^{10/3}(a^{8/3}-t^{8/3})\frac{\omega(a\rho)^2}{f(t)^2}\,. 
\]
Hence we get
\begin{equation}\label{finalestimatef}
\frac{t^3}{f(t)} + c a^{8/3} t^{-8/3}\omega(a\rho)^2 \frac{t^6}{f(t)^2} \ \ge\ M.
\end{equation}

Eventually, we observe that
\[
f(t)=\int_{D_t^t}\left(1-\frac{z(\rho y)\cdot\onu}{F(\onu)}\right)\,dy = \frac{1}{\rho^d}\int_{D_{\rho t}^{\rho t}} \left(1-\frac{z(x)\cdot\onu}{F(\onu)}\right)\,dx\,,
\]
so that \eqref{finalestimatef} can be rewritten
\begin{equation}\label{almostdone}
\left(\frac{\displaystyle \int_{D_{\rho t}^{\rho t}}
\left(1-\frac{z(x)\cdot\onu}{F(\onu)}\right)dx}{(\rho t)^3}\right)^{-1}
\ \ge\ 
\frac{-1+\sqrt{ 1+4 M c a^{8/3} t^{-8/3} \omega(a\rho)^2}}
{2c a^{8/3}t^{-8/3} \omega(a\rho)^2}
\end{equation}
The value of $t$ being fixed, we can choose the value of $\rho$ small enough
in order to have
$4 M c a^{8/3} t^{-8/3} \omega(a\rho)^2 < 1$, and (using
$\sqrt{1+X}\ge 1+X/2 - X^2/8$ if $X\in (0,1)$), \eqref{almostdone}
yields
\begin{equation}\label{almostdone2}
\left(\frac{\displaystyle \int_{D_{\rho t}^{\rho t}}
\left(1-\frac{z\cdot\onu}{F(\onu)}\right)\, dy}{(\rho t)^3}\right)^{-1}
\ \ge\  M - M^2 ca^{8/3} t^{-8/3}\omega(a\rho)^2\ \ge\ \frac{3}{4}M\,.
\end{equation}
It  follows that
\begin{equation}
\limsup_{\e\to 0} \frac{\displaystyle \int_{D_{\e}^{\e}} \left(1-\frac{z\cdot\onu}{F(\onu)}\right)\,dy}{\e^3}
\ \le\ \frac{4}{3}M^{-1}
\end{equation}
and since $M$ is arbitrary, $0$ is indeed a Lebesgue point of $z$, with
value $\bar z=\nabla F(\onu)$ (recall that $1-\frac{z(x)\cdot\onu}{F(\onu)} \ge (C/F(\onu)) |z(x)-\bar z|^2$).

{\it Step 2.} When $F$ depends also on the $x$ variable, the proof follows along the same lines 
as in {\it Step 1}, taking into account the errors terms in \eqref{bigtwo} and \eqref{big4}. 
Keeping the same notations as in {\it Step 1}
and setting $\bar z:= \nabla_p F(0,\bar \nu)$ we find that since  $F^\circ(0,\bar z)\le 1$, there holds $\bar z\cdot \nu^{E_\rho}\le F(0,\nu^{E_\rho})$ and thus
\[\int_{\partial E_\rho \cap D_s^t} \bar z\cdot \nu^{E_\rho}-F(\rho x, \nu^{E_\rho}) d\H^{d-1}\le\int_{\partial E_\rho \cap D_s^t} |F(0,\nu^{E_\rho})-F(\rho x,\nu^{E_\rho})| d\H^{d-1}
\le C \rho s^{d-1}
\]
where the last inequality follows from $t\le s$ and the minimality of $E_\rho$ inside $D_s^t$. Now since 
\[\left( F^\circ\right)^2(0,z_\rho)-\left( F^\circ\right)^2(\rho x,z_\rho)\ge \left( F^\circ\right)^2(0,z_\rho)-1\ge 2 \frac{\bar \nu}{F(0,\bar \nu)} \cdot (z_\rho-z) +\delta^2|z_\rho-z|^2\]
we find that \eqref{star} transforms into, 
\[(\xi\cdot z_\rho)^2\le C \left[(F(0,\bar \nu)-\bar \nu\cdot z_\rho) +(\left( F^\circ\right)^2(0,z_\rho)-\left( F^\circ\right)^2(\rho x,z_\rho))\right]\]
for every $|\xi|\le 1$ and $\xi\cdot \bar z=0$, from which we get 
\begin{multline*}
 \int_{\partial D_s+(-t,t)\bar z}
\left((2\car{E_\rho} -1) t -\frac{\onu\cdot y}{F(\onu)} \right)
\,\left(z_\rho\cdot \xi\right)\,d\H^{d-1}
\\ \le\,
2CF(0,\onu) t\sqrt{t} \left(
\int_{\partial D_s+(-t,t)\bar z} 
 \left(1-\frac{z_\rho\cdot\onu}{F(0,\onu)}\right) \,d\H^{d-1}
\right)^\half \sqrt{\H^{d-2}(\partial D_s)}\\
+ 2Ct \,\int_{\partial D_s+(-t,t)\bar z}
 \left|\left( F^\circ\right)^2(0,z_\rho)-\left( F^\circ\right)^2(\rho x,z_\rho)\right|^{1/2}\, d\H^{d-1}\\
\le \,C F(0,\onu) t\sqrt{t} \left(
\int_{\partial D_s+(-t,t)\bar z} 
 \left(1-\frac{z_\rho\cdot\onu}{F(0,\onu)}\right) \,d\H^{d-1}
\right)^\half \sqrt{\H^{d-2}(\partial D_s)}+ Ct \rho^{1/2} s^{d-1} t\,.
\end{multline*}
Using these estimates, we finally get that, setting as before $f(s):=\int_{D^t_s} \left(1-\frac{z_\rho\cdot\onu}{F(0,\onu)}\right)dy$, there holds
\[f(s)^2\ \le\  c\left( s^{d-2} t^3  f'(s)
\,+\,t^{4-2/d} s^{2d-4+2/d} \omega(\rho s)^2+ \rho t s^{d-1}+\rho^{1/2} t^2 s^{d-1}\right)\,.\]
{}From this inequality, the proof can be concluded exactly as in {\it Step 1}.
\end{proof}

\begin{remark}\label{remDu}\rm
Assuming $F$ has some regularity (Lipschitz in the first variable,
and $C^{2,\beta}$ and even in the second, see~\cite{SchoenSimon}), then for
$d=2$ or $d=3$ and $g\in L^p(\Om)$ with $p>d$, $\partial E$ is of class $C^{1,\alpha}$ for some $\alpha>0$. In this case, \eqref{euleru} holds everywhere in $\supp(Du)$. 
\end{remark}

Eventually, we can also give a locally uniform convergence result (valid
in dimension $d=2,3$, with the assumption\footnote{Probably just technical.} that $F$ is even in dimension 3).
\begin{proposition}
For all $x\in\Om$ we let $$z_\rho(x):=\frac{1}{|B_\rho(0)|}\int_{B_\rho(x)\cap\Om}z \,dy\,.$$
Then, $F^\circ(x,z_\rho(x))\to 1$ locally uniformly on $S$.
\end{proposition}
\begin{proof}
Given $K\subset \Om$ a compact set, we can check that 
for any $t>0$, there exists $\rho_0>0$ such that for any $x\in K\cap S$,
if $E^x$ is the level set of $u$ through $x$, then 
for any $\rho\le\rho_0$, the boundary of
$(E^x-x)/\rho\cap B_1(0)$ lies in a strip of width $2t$, that is, there is
$\onu^x_\rho\in \Sp^{d-1}$ with $\partial((E^x-x)/\rho)\cap B_1(0)\subset
\{|y\cdot\onu^x_\rho |\le t\}$).

Indeed, if this is not the case, one can find $t>0$, $\rho_k\to 0$, $x_k\in K\cap S$,
such that $\partial((E^{x_k}-x_k)/\rho_k)\cap B_1(0)$ is not contained
in any strip of width $2t$. Up to a subsequence we may assume
that $x_k\to x\in K\cap S$, and from the bound on the perimeter, that $(E^{x_k}-x_k)/\rho_k$ converges (in the Kuratowski sense) to a local minimizer of $\int_{\partial E} F(0,\nu^E) d\H^{d-1}$
 and is thus a halfspace.\footnote{If $d=2$, this Bernstein result readily follows from the strict convexity of $F$, see \cite[Prop 3.6]{CGN} 
whereas for $d=3$, see \cite[Thm.~4.1]{white}, where it is assumed that $F$ is even.
In the case of the area i.e when $F(x,Du)=|Du|$ and $d\le 7$, see also \cite[Rem 3.2]{GoMa}.} It yields that $\partial((E^{x_k}-x_k)/\rho_k)\cap B_1(0)$
converges in the Hausdorff sense (thanks to the density estimates) to a
hyperplane. We easily obtain a contradiction.

The thesis follows when we observe that the proof of Proposition~\ref{propLebesgue}
can be reproduced by replacing 
the direction $\nu^{E^x}(x)$ (which exists only if $x$ lies in the reduced boundary of $E^x$)
with the direction $\onu^x_\rho$ given above.
\end{proof}
\begin{remark}\rm
In the Euclidean case ($F=|\,.\,|$) it has already been observed in~\cite{AmbPao}
that the blow-ups are flat at each point of the boundary of a set
with curvature in $L^d$ (but for a closed set of maximal dimension $d-8$),
however the spiral example in~\cite{GonzalezMassariTamanini} shows that
even if $d=2$, the orientation of the limit line may not be unique.
\end{remark}
\subsection{A counterexample.}
We provide an example where $g\in L^{d-\e}(\Om)$, with $\e>0$ arbitrarily small, 
and Theorem \ref{propLebesgue} does not hold. 

Let $\Om=B_1(0)$ be the unit ball of $\R^d$ and let $E=\Om\cap\{ x_d\le 0\}$. We shall construct a vector
field $z:\Om\to \R^d$ such that $z=\nu^E$ on $\partial E\cap \Om$, $|z|\le 1$ everywhere in $\Om$, ${\rm div}z\in L^{d-\e}(\Om)$,
but $0$ is not a Lebesgue point of $z$. Notice that $E$ minimizes the functional \eqref{geomper} with $g={\rm div}z$.

Letting $r_n\to 0$ be a decreasing sequence to be determined later, and let $B_n=B_{r_n}(x_n)$ with $x_n=2 r_n e_d$.
Without loss of generality, we may assume $r_{n+1}<r_n/4$ so that the balls $B_n$ are all disjoint.
We define the vector field $z$ as follows: $z(x)=e_d$ if $x\in\Om\setminus\cup_n B_n$, and
$z(x)=|x-x_n|e_d$ if $x\in B_n$. 
It follows that ${\rm div}z=0$ in $\Om\setminus\cup_n B_n$ and $|{\rm div}z|\le 1/r_n$ in $B_n$, so that 
\[
\int_\Om |{\rm div}z|^{d-\eps} \,dx = \sum_n \int_{B_n} |{\rm div}z|^{d-\eps} \,dx \le \omega_d \sum_n r_n^\eps<+\infty
\]
if we choose $r_n$ converging to zero sufficiently fast, so that $g=-\Div z\in L^{d-\eps}(\Om)$.

However, since $z\cdot e_d\le 1/2$ in $B_{r_n/2}(x_n)$, we also have
\[
\int_{B_{3r_n}(0)}z\cdot e_d\,dx \le \left|B_{3r_n}(0)\right| -\frac 12 \left|B_{r_n/2}(x_n)\right| 
\]
so that 
\[
\frac{1}{\left|B_{3r_n}(0)\right|}\, \int_{B_{3r_n}(0)}z\cdot e_d\,dx
\le 1-\frac{1}{6^d}<1\,.
\]
On the other hand, for $\delta\in (0,1/6^d)$ we have
\[
\frac{1}{\left|B_{r_n}(0)\right|}\,
\int_{B_{r_n}(0)}z\cdot e_d\,dx \ge 
\frac{1}{\left|B_{r_n}(0)\right|}\left(
\left|B_{r_n}(0)\right| - \sum_{i=n+1}^\infty \left|B_{r_i}(x_i)\right|\right)\ge 1-\delta\,,
\]
if we take the sequence $r_n$ converging to $0$ sufficiently fast.
It follows that $0$ is not a Lebesgue point of $z$.


\end{document}